\documentclass{article}
\usepackage{geometry}
\usepackage{tikz, subfigure, amsmath, amsthm, bm, amssymb, latexsym, eucal, graphicx, epsfig,bm, fancyhdr,pgfplots, setspace}
\usepackage[square, sort, comma, numbers]{natbib}
\numberwithin{equation}{section}
\numberwithin{figure}{section}

\newtheorem{theorem}{Theorem}
\newtheorem{lemma}{Lemma}
\newtheorem{proposition}{Proposition}

\newtheorem{corollary}{Corollary}

\renewcommand{\tilde}{\widetilde}

\everymath{\displaystyle}
\usetikzlibrary{shapes,arrows}
\newenvironment{condition}[1]
{ \begin{equation*}
	\mbox{\raisebox{.5pt}{\textcircled{\raisebox{-.9pt}{\textbf{#1}}} \quad}}}
{
	\end{equation*}
}
\begin{document}
\tikzstyle{block} = [draw,fill=blue!20,rectangle,minimum height=3em, minimum width=3em]
\tikzstyle{input} = [coordinate]
\tikzstyle{output} = [coordinate]
\tikzstyle{pinstyle} = [pin edge={->,thin,black,dashed}]
\tikzstyle{pinstyle2} = [pin edge={<-,thin,black,dashed}]
\pagestyle{fancy}
\thispagestyle{plain}
\title{The Dynamics of Vector-Borne Relapsing Diseases
}

\author{Cody Palmer         \and
Erin Landguth \and
Emily Stone
}

\maketitle

\begin{abstract}
	In this paper we describe the dynamics of a vector-borne relapsing disease, such as tick-borne relapsing fever, using the methods of compartmental models. 
	After some motivation, model description, and a brief overview
	of the theory of compartmental models, we compute a general form of the reproductive ratio $R_0$, which is the average number of new infections produced by a single infected individual.
	A disease free equilibrium undergoes a bifurcation at $R_0 =1$ and we show that for an arbitrary number of relapses it is a transcritical bifurcation with a single branch of
	endemic equilibria that is locally asymptotically stable for $R_0$ sufficiently close to 1.  We then show that these results can be extended to a variant of the model that allows
	for variation in the number of relapses before recovery.  We close with some discussion and directions for future research.
\end{abstract}

\section{Introduction}
\label{intro}
Many mathematical models dealing with the spread of infectious diseases show a rich variety of dynamics that arise from various nonlinear interactions or temporal forcing e.g. \cite{keeling2001seasonally}.
Vector-borne diseases are additionally complex with interactions between host and vector species \cite{dobson2006sacred}.
Tick-borne relapsing fever (TBRF) is an example
of a system that incorporates such complex interactions in a multiple host-vector community.

In North America, TBRF is caused by several species of spiral-shaped bacteria (\textit{Borrelia} spp.) that are transmitted to their hosts through the bite
of an infected vector, the soft ticks of the genus \textit{Ornithodoros}.  Once infected with the bacteria, ticks remain infectious for extended periods and possibly
for life \cite{johnsonPLOSNTD}.  Most human cases occur in the summer months and are often associated with sleeping in rustic cabins in mountainous areas of the Western United States
\cite{dworkin1998tick}.  The model presented in this paper
is motivated by a system located on Wild Horse Island, Flathead Lake, Lake County, Montana (WHI), where the presence of this pathogen has been
confirmed \cite{schwan2003tick}.  The island harbors two host species, the red squirrel (\textit{Tamiasciurus hudsonicus}) and the deer mouse (\textit{Peromyscus maniculatus}) and a single vector species
(\textit{O. hermsi}), which is thought to control the disease patterns on the island.  See \cite{johnsonPLOSNTD} for more
details.

Compartmental models, such as the SIR models with susceptible, infectious, and removed compartments, have been applied to many disease and disease-like systems in an effort to examine
system dynamics \cite{kermack1927contribution,feng2015stability, lin2016modeling,white2007heroin,okuneye2016analysis,xu2016voluntary,saad2016mathematical, guo2006global}.  
In these epidemic models, susceptible individuals pass into the infected class and then transition to the removed class.  For some diseases,
recovered individuals may relapse through a reactivation of infection and revert back to an infected class.  TBRF is a system in which relapse always occurs,
but between different infected classes caused by the bacteria's antigenic variation \cite{antigen}.  The advantage of antigenic variation is to extend the length of infection
so that the host will still be infected at the next interaction with a susceptible vector \cite{barry1991dynamics,gupta2005parasite}.  The question that we raise is how the dynamics of the
spread of the disease are affected by the number of relapses, and how do these dynamics differ from a vector-borne disease with no relapses.

Given a mathematical model for disease spread, the disease reproduction number, $R_0$, is an essential summary parameter.  It is defined as the average number of secondary infections produced
when one infected individual is introduced into a host population in which all individuals are susceptible \cite{dietz1975transmission}.  When $R_0 < 1$, the disease free equilibrium (DFE), at which the population
remains in the absence of disease, is locally asymptotically stable.  However, if $R_0>1$, then the DFE is unstable and invasion is always possible
\cite{hethcote2000mathematics}
and a new endemic equilibrium (EE) exists. 

 A key
assumption for the host-vector disease modeling is the definition of the transmission term, which represents the contact between hosts and vectors.  The formulation of the transmission
term directly affects the reproduction number $R_0$.  For host-vector disease models, the transmission term includes vector
biting rate $f$, which controls the disease transmission both from the vector-to-host and from the host-to-vector.  The TBRF model follows frequency-dependent transmission assumptions
through the biting rate, since a blood meal is required only once every three months regardless of the host population density.  Following this framework, it is reasonable
to assume that a host would experience an increasing number of bites as the vector population increased \cite{johnsonPLOSNTD}.  While our work here shares techniques with previous work 
done on staged progression models \cite{guo2006global, hyman1999differential, hyman2005reproductive}, the key difference is the addition of vectors.

In this paper we derive a general form for  $R_0$ when there is a single host species in the system, following the
methodology of van den Dreissche and Watmough for general compartmental disease models \cite{compart} which
is then extended with an arbitrary number of relapsing states.  From this
we show how $R_0$ depends on the number of relapses and the various parameters in the model.  We then classify the bifurcation at $R_0=1$, showing that it is transcritical with an exchange
of stability between a disease free equilibrium and an endemic equilibrium.  We also show that there is a unique endemic equilibrium for each value of $R_0 > 1$.  We finally consider
a variation of the model which accounts for differing number of relapses before recovery, and close with discussion and future work.

\section{Single Host Vector Model}
\label{sec:1}
To begin constructing the model we first make assumptions motivated by the spread of TBRF on WHI.  We assume that new infections only occur
when an infected vector bites a susceptible host or when a susceptible vector bites an infected host.  We also assume that when a vector becomes infected, it is infected for life.  Furthermore, we assume that the transmission terms are frequency dependent through the biting rate $f$.  The infected hosts relapse into infected compartments sequentially at a rate $\alpha_i$ and recover from
the disease at rate $\gamma$.  The total populations of hosts and vectors are assumed to remain constant are denoted
by $N$ and $\widetilde{N}$ respectively (throughout the
paper we will indicate quantities corresponding to the  vectors with a $\widetilde{}$ ).  Though the death rates may vary
among the different infected compartments, we will require that recruitment (reproduction) and death rates are equal.  In more precise terms,
we assume that the dynamics occur on an invariant hyperplane.  This assumption greatly simplifies the model, as will be seen shortly.
Although it appears at first to be a restrictive assumption, it is a natural for determining $R_0$, where it is often assumed that the disease is spreading among a population large enough that the net change is essentially zero.

The infection dynamics in a single host-vector system with $j-1$ relapsing rates for $j\geq 1$ infected compartments involve the
number of susceptible hosts $S(t)$, infectious hosts $I_k(t)$, removed hosts $R(t)$, susceptible vectors $\widetilde{S}(t)$, and infected vectors $\widetilde{I}(t)$,
where the total host population is $\displaystyle{N = S+\sum_{k=1}^j I_k + R}$ and the total vector population $\widetilde{N} = \widetilde{S} + \widetilde{I}$.  A conceptual
model for this scheme is given in Figure \ref{fig:1}.
\begin{figure}
	\centering
	\subfigure[0 relapses.]{\begin{tikzpicture}[auto, node distance=2cm,>=latex']
		\node [block,pin={[pinstyle,pin distance=.5cm]above:$\mu_s$},pin={[pinstyle2,pin distance=.5cm]left:$\beta(N)$}] (sh) {$S$};
		\node [block,right of=sh,pin={[pinstyle,pin distance=.5cm]above:$\mu_1$}] (ih1) {$I_1$};
		\node [block,right of=ih1,pin={[pinstyle,pin distance=.5cm]above:$\mu_r$}] (r) {$R$};

		\draw [draw,->] (sh) -- node[name=u] {} (ih1);
		\node [block,below of=u,pin={[pinstyle,pin distance=.5cm]below:$\widetilde{\mu_{i}}$}, node distance=2cm] (iv) {$\widetilde{I}$};
		\draw [draw,->] (iv) -- node {$fc_v$} (u);
			
		\draw [draw,->] (ih1) -- node[name=gam] {$\gamma$} (r);
		\node [block,below of=gam,pin={[pinstyle,pin distance=.5cm]below:$\widetilde{\mu_{s}}$}, pin={[pinstyle2,pin distance=.5cm]right:$\beta_v(\widetilde{N},N)$},
			node distance=2.1cm] (sv) {$\widetilde{S}$};
		\draw [draw,->] (sv) -- node[name=mid] {} (iv);
		\draw [draw,->] (ih1) -- node {$fc$} (mid);
	\end{tikzpicture}}

	\subfigure[1 relapse]{
	\begin{tikzpicture}[auto, node distance=2cm,>=latex']
		\node [block,pin={[pinstyle,pin distance=.5cm]above:$\mu_s$},pin={[pinstyle2,pin distance=.5cm]left:$\beta(N)$}] (sh) {$S$};
		\node [block,right of=sh,pin={[pinstyle,pin distance=.5cm]above:$\mu_1$}] (ih1) {$I_1$};
		\node [block,right of=ih1,pin={[pinstyle,pin distance=.5cm]above:$\mu_2$}] (ih2) {$I_2$};
		\node [block,right of=ih2,pin={[pinstyle,pin distance=.5cm]above:$\mu_r$}] (r) {$R$};

		\draw [draw,->] (sh) -- node[name=u] {} (ih1);
		\node [block,below of=u,pin={[pinstyle,pin distance=.5cm]below:$\widetilde{\mu_{i}}$}, node distance=2cm] (iv) {$\widetilde{I}$};
		\draw [draw,->] (iv) -- node {$fc_v$} (u);
		\draw [draw,->] (ih1) -- node[name=a] {$\alpha_1$} (ih2);
			
		\draw [draw,->] (ih2) -- node[name=gam] {$\gamma$} (r);
		\node [block,below of=gam,pin={[pinstyle,pin distance=.5cm]below:$\widetilde{\mu_{s}}$}, pin={[pinstyle2,pin distance=.5cm]right:$\beta_v(\widetilde{N},N)$},
			node distance=2.1cm] (sv) {$\widetilde{S}$};
			\draw [draw,->] (sv) -- node[name=mid1] {} (iv);
		\draw [draw,->] (ih1) -- node {$fc$} (mid);
		\node [below of=ih2] (blow) {};
		\draw [draw,->] (ih2) -- node {$fc$} (blow);
\end{tikzpicture}}

\subfigure[$j-1$ relapses]{
	\begin{tikzpicture}[auto, node distance=2cm,>=latex']
		\node [block,pin={[pinstyle,pin distance=.5cm]above:$\mu_s$},pin={[pinstyle2,pin distance=.5cm]left:$\beta(N)$}] (sh) {$S$};
		\node [block,right of=sh,pin={[pinstyle,pin distance=.5cm]above:$\mu_1$}] (ih1) {$I_1$};
		\node[right of=ih1] (bla) {$\ldots$};
		\node [block,right of=bla,pin={[pinstyle,pin distance=.5cm]above:$\mu_j$}] (ihj) {$I_j$};
		\node [block,right of=ihj,pin={[pinstyle,pin distance=.5cm]above:$\mu_r$}] (r) {$R$};

		\draw [draw,->] (sh) -- node[name=u] {} (ih1);
		\node [block,below of=u,pin={[pinstyle,pin distance=.5cm]below:$\widetilde{\mu_{i}}$}, node distance=2cm] (iv) {$\widetilde{I}$};
		\draw [draw,->] (iv) -- node {$fc_v$} (u);
		\draw [draw,->] (ih1) -- node[name=a] {$\alpha_1$} (ih2);
			
		\draw [draw,->] (ihj) -- node[name=gam] {$\gamma$} (r);
		\node [block,below of=gam,pin={[pinstyle,pin distance=.5cm]below:$\widetilde{\mu_{s}}$}, pin={[pinstyle2,pin distance=.5cm]right:$\beta_v(\widetilde{N},N)$},
			node distance=2.1cm] (sv) {$\widetilde{S}$};
			\draw [draw,->] (sv) -- node[name=mid1] {} (iv);
		\draw [draw,->] (ih1) -- node {$fc$} (mid);
		\node [below of=ihj] (blow) {};
		\draw [draw,->] (ihj) -- node {$fc$} (blow);
		\draw [draw,->] (bla) -- node {$\alpha_{j-1}$} (ihj);
\end{tikzpicture}}
\caption{Conceptual models for the cross-infection dynamics between a single host-vector system,
	which includes (a) no relapses between $j=1$ infected compartments, (b) 1 relapse between $j=2$ infected compartments, and (c)
	$j-1$ relapses between $j$ infected compartments.  Dashed lines are the vital rates for each population, where solid lines refer to
interaction rates between compartments.}
	\label{fig:1}
\end{figure}
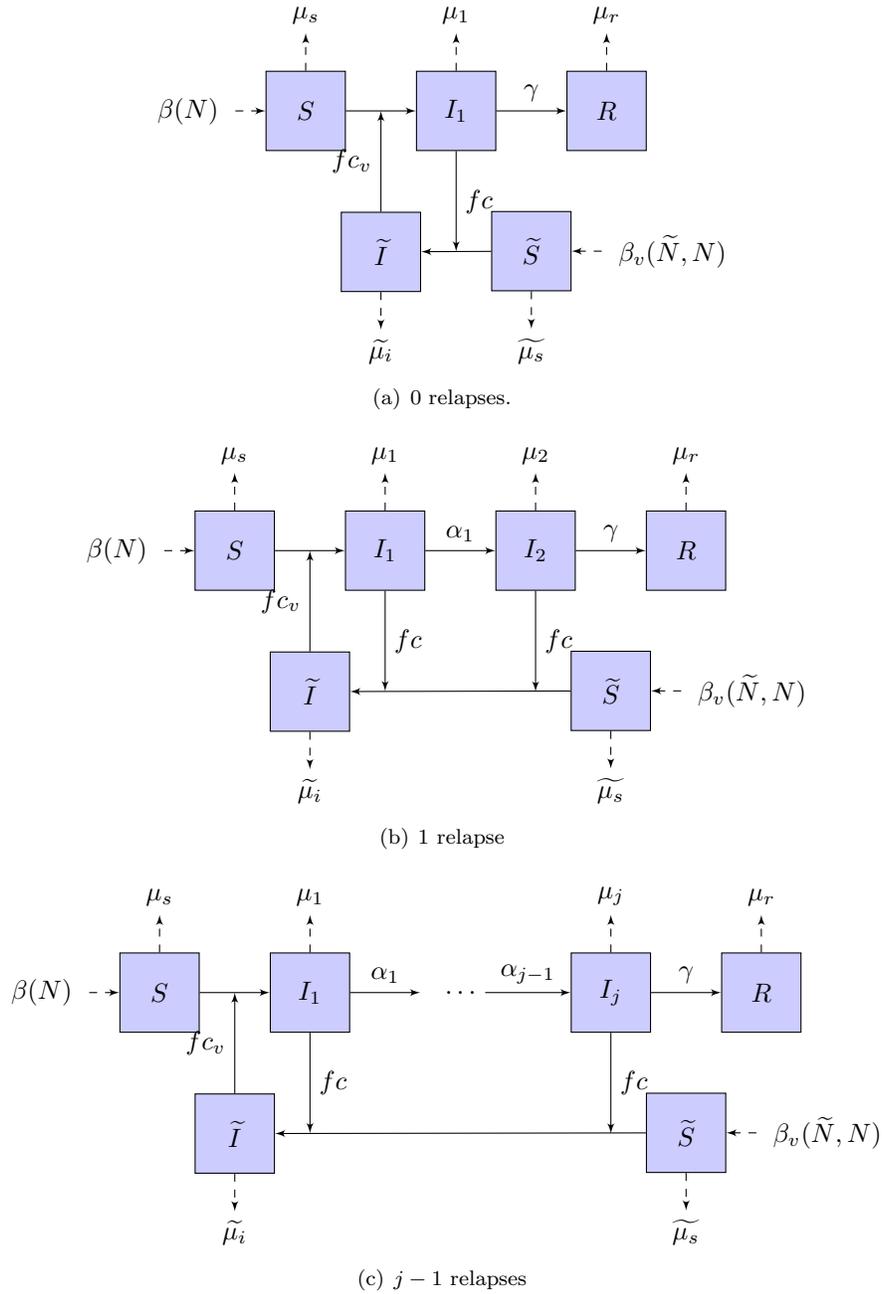
The equations for the model are as follows: first the host equations
\begin{equation}
	\begin{array}{ccl}
		S' &=& \beta(N) - fc_v\widetilde{I} \frac{S}{N} - \mu_sS,\\
		I'_1 &=& fc_v\widetilde{I} \frac{S}{N} - \alpha_1I_1 - \mu_1I_1,\\
		I'_2 &=& \alpha_1I_1-\alpha_2I_2 - \mu_2I_2,\\
		 & \vdots& \\
		I'_{j-1} &=& \alpha_{j-2}I_{j-2} - \alpha_{j-1}I_{j-1} - \mu_{j-1}I_{j-1},\\
		I'_j &=& \alpha_{j-1}I_{j-1} - \gamma I_j - \mu_jI_j,\\
		R' &=& \gamma I_j - \mu_rR,
	\end{array}
	\label{host}
\end{equation}
and the vector equations:
\begin{equation}
	\begin{array}{ccl}
		\widetilde{S}' &=& \beta_v(\widetilde{N},N) - \frac{fc\widetilde{S}}{N} \sum_{k=1}^j I_k - \widetilde{\mu}_s\widetilde{S},\\
		\widetilde{I}' &=& \frac{fc\widetilde{S}}{N} \sum_{k=1}^j I_k -\widetilde{\mu}\widetilde{I}.
	\end{array}
	\label{vec}
\end{equation}
\begin{table}[t]
	\begin{center}
\begin{tabular}{|c|c|c|}
\hline
Variable & Definition & Dimensionless Form\\
\hline
$N$ & Total host population. & $n$ \\
\hline
$\tilde{N}$ & Total vector population. & $\tilde{n}$ \\
\hline
$S$ & Susceptible hosts. & $s$\\
\hline
$I_k$ & Hosts in the $k$\textsuperscript{th} infected compartment. & $i_k$\\
\hline
$R$ & Recovered hosts. & $r$\\
\hline
$\tilde{S}$ & Susceptible vectors. & $\tilde{s}$\\
\hline
$\tilde{I}$ & Infected vectors. & $\tilde{i}$\\
\hline
\end{tabular}
\vskip.1in
\begin{tabular}{|c|c|c|}
\hline
Parameter & Definition & Dimensionless Form\\
\hline
$f$ & Biting Rate. & \\
\hline
$c_v$ & Vector Competency. & \\
\hline
$\mu_s$ & Susceptible Death Rate. & $b_s$\\
\hline
$\alpha_k$ & Transfer rate from the $k$\textsuperscript{th} infected compartment. & $q_k$\\
\hline
$\mu_k$ & Death rate in the $k$\textsuperscript{th} infected compartment. & $b_k$\\
\hline
$\gamma$ & Rate of recovery. & 1\\
\hline
$\mu_r$ & Death rate of recovered individuals. & $b_r$\\
\hline
$c$ & Host competency.& \\
\hline
$\tilde{\mu}_s$ & Death rate of susceptible vectors. & $\tilde{b}_s$\\
\hline
$\tilde{\mu}$ & Death rate of infected vectors. & $\tilde{b}$\\
\hline
	\end{tabular}

	\caption{Variables and parameters for the model, with dimensionless forms where relevant.  All parameters are positive.}
	\label{defs}
\end{center}
\end{table}
 The growth rates $\beta$ and $\beta_v$ are logistic, given as follows:
\[
	\beta(N) = \beta_1N - \left( \frac{ \beta_1 - \mu_s}{\overline{S}} \right)N^2,
\]
\[
	\beta_v(\widetilde{N},N) = \beta_{v1}\widetilde{N}N - \left( \frac{\beta_{v1}\overline{S}-\widetilde{\mu}_s}{\overline{S}_v} \right) \widetilde{N}_v^2.
\]
We insist that $\beta_1\geq \mu_s$ and $\beta_{v1} \geq \tilde{\mu}_s$, for these to be well defined.  Further details regarding this being a well defined model
are found in Appendix \ref{welldef}, where we show that the system satisfies the conditions laid out in \cite{compart}.
Here $\overline{S}$ and $\overline{S}_v$ are constants.  It is then easy to see that
\[
	(S,I_1,  \ldots, I_j, R,\widetilde{S}, \widetilde{I}) = (\overline{S}, 0, \ldots, \overline{S}_v, 0)
\]
is a fixed point of the system.  This is known as the Disease Free Equilibrium (DFE).  To investigate the stability of the DFE we calculate $R_0$ for arbitrary $j$.
We will make the natural assumptions in this model that $\mu_i \geq \mu_s$, $\mu_r \geq \mu_s$, and that $\tilde{\mu} \geq \tilde{\mu}_s$.  We are assuming
that having the disease will only serve to increase the death rate over the susceptible population.  This leads to the following result:
\begin{proposition}\label{eqdeath}
	The manifold $N = \overline{S}$ is invariant if and only if $\mu_s=\mu_i$ and
	$\mu_s = \mu_r$
\end{proposition}
\begin{proof}
First suppose that $\mu_s=\mu_i$ and $\mu_s = \mu_r$. Substituting this into the system and summing the equations yields
\[
	N' = \beta(N) - \mu_sN = \beta_1N - \frac{(\beta_1-\mu_s)N^2}{\overline{S}} - \mu_sN = (\beta_1 - \mu_s)N\left( 1- \frac{N}{\overline{S}}\right)
\]
From this we can see if $N(0) = \overline{S}$, then  $N(t) = \overline{S}$, and $N=\overline{S}$ is invariant.

Now assume that $N = \overline{S}$ is invariant.  Then if $N(0) = \overline{S}$,  $N$ is constant and the first equation reduces to
\[
	S' = \mu_s(\overline{S} - S) - \frac{fc_v\widetilde{I}S}{\overline{S}}
\]
Then, summing the equations gives that
\[
	N' = \mu_sS + \sum_{k=1}^j \mu_kI_k + \mu_rR - \mu_s \overline{S}
\]
But since $N$ is constant, $N' = 0$, and we use the fact that $s + \sum I_k + R = \overline{S}$ so the above equation becomes
\[
	0 = \sum \left( \frac{\mu_i}{\mu_s} - 1 \right)I_k + \left( \frac{\mu_r}{\mu_s} - 1\right)R
\]
Now, $I_k,R \geq 0$, and since $\mu_i \geq \mu_s$ and $\mu_r \geq \mu_s$ we have
\[
	\frac{\mu_i}{\mu_s} \geq 1  \mbox{ and } \frac{\mu_r}{\mu_s} \geq 1
\]
Since each term is nonnegative, the only way the sum can add to zero is if each term is zero.  Since $I_k$ and $R$ can, at some time, be nonzero
we must have $\frac{\mu_i}{\mu_s} - 1 = 0$ and $\frac{\mu_r}{\mu_s} - 1 = 0$.  Thus  $\mu_s = \mu_i = \mu_r$.
\end{proof}
When the death rates are equal the host population has simple logistic dynamics:
\[
N'= (\beta_1 - \mu_s)N\left( 1- \frac{N}{\overline{S}}\right)
\]
In particular, if $\beta_1 > \mu_s$ we have $N=\overline{S}$ is an attracting fixed point.
We can then apply the results from Chapter 2 of \cite{contnon}:
a fixed point that is asymptotically stable on the
manifold $N=\overline{S}$ is asymptotically stable in the off-manifold dynamics.  This will be relevant to
our discussion of endemic equilibria in Section \ref{bifsec}.  The equality of the death rates is equivalent to the invariance of the manifolds $N = \overline{S}$ and $\tilde{N} = \overline{S}_v$,
and restriction to these manifolds yields a simplification of the growth terms:
\[
	\beta(\overline{S}) = \mu_s\overline{S} \mbox{ and }   \beta_v(\overline{S}_v, \overline{S}) = \tilde{\mu}_s\overline{S}_v
\]

\section{$R_0$ for the single host-vector system with $j-1$ relapses.}
\label{sec:3}
\subsection{Dimensionless Form}
\label{3sec:1}
To ease some calculation we will put equations (\ref{host}) and (\ref{vec}) in dimensionless form.  Letting $\tau = \gamma t$, and scaling all the population variables by
the corresponding initial total populations $N(0)$ and $\widetilde{N}(0)$ gives the dimensionless form.
\[
	\frac{ds}{d\tau} = \frac{\beta(N)}{\gamma N(0)} - k\tilde{i} \frac{s}{n} - b_s s
\]
\[
	\frac{di_i}{d\tau} = k\tilde{i} \frac{s}{n} - q_1i_1 - b_{1}i_1
\]
\[
	\frac{di_2}{d\tau} = q_1i_1 - q_2i_2 - b_{2}i_2
\]
\[
	\vdots
\]
\[
	\frac{di_{j-1}}{d\tau} = q_{j-2}i_{j-2} - q_{j-1}i_{j-1} - b_{(j-1)}i_{j-1}
\]
\[
	\frac{di_j}{d\tau} = q_{j-1}i_{j-1} - i_j - b_{j}i_j
\]
\[
	\frac{dr}{d\tau} = i_j-b_rr
\]
\[
	\frac{d\tilde{s}}{d\tau} = \frac{\beta_v(\tilde{N}, N)}{\gamma \tilde{N}(0)} -\frac{l\tilde{s}}{n} \sum_{m=1}^j i_m -\tilde{b}_{s}s_v
\]
\[
	\frac{d\tilde{i}}{d\tau} = \frac{l\tilde{s}}{n} \sum_{m=1}^j i_m - \tilde{b}_{i} \tilde{i}
\]
where $k=fc_v/\gamma$ and $l=fc/\gamma$.
\subsection{General form for $R_0$}
\label{3sec:2}   Following \cite{compart}, we consider the reduced equations
\[
\frac{d}{d\tau} \begin{bmatrix} i_1\\ i_2\\ \vdots \\ i_{j-1}\\i_j\\\tilde{i} \end{bmatrix} =
\begin{bmatrix} k\tilde{i} \frac{s}{n}\\ 0\\ \vdots\\ 0\\0\\ \frac{l\tilde{s}}{n} \sum_{k=1}^j i_k \end{bmatrix}
	 - \begin{bmatrix} q_1i_1 + b_{1}i_1\\ -q_1i_1+q_2i_2 + b_{2}i_2\\ \vdots\\
		 -q_{j-2}i_{j-2} + q_{j-1}i_{j-1} + b_{j-1}i_{j-1}\\ -q_{j-1}i_{j-1} + i_j + b_{j}i_j\\
	 \tilde{b}_{i}\tilde{i} \end{bmatrix}
= \bm{w}-\bm{v}
 \]
 Next we take the Jacobian of $\bm{w}$ and $\bm{v}$ and evaluate them at the disease free equilibrium in order to find the matrices $W$ and $V$, i.e.
 \[
	 W = \left. \begin{bmatrix} 0 & \ldots & 0 & k\frac{s}{n}\\
	 			0& \ldots & 0 & 0\\
	 \vdots &  & \vdots & \vdots\\
	 \frac{l\tilde{s}}{n}& \ldots& \frac{l\tilde{s}}{n} & 0
 \end{bmatrix} \right|_{DFE}
 \]
 At the DFE we note that the entire host and vector populations are susceptible and there are no hosts in
 any of the relapse states.  This means that $n=s = \overline{s}$ and $\tilde{n} = \tilde{s}= \overline{s_v}$ (the carrying
 capacities of each population).  Hence
 \[
	 W = \begin{bmatrix} 0 & \ldots & 0 & k\\
	 			0& \ldots & 0 & 0\\
	 \vdots &  & \vdots & \vdots\\
	 \frac{l\overline{s_v}}{\overline{s}}& \ldots& \frac{l\overline{s_v}}{\overline{s}} & 0
 \end{bmatrix}
 \]
 For $V$, we note that the Jacobian is made up of constant values, namely
 \[
	V = \begin{bmatrix} q_1+b_{1} & 0 & 0 & \ldots & 0 & 0 & 0 &0\\
	 			-q_1 & q_2+b_{2} & 0 & \ldots& 0 & 0 & 0 &0\\
 				0 & -q_2 & q_3+b_{3}&\ldots& 0 & 0 & 0 &0\\
	 \vdots &\vdots &\vdots & & \vdots &\vdots &\vdots &\vdots \\
	 0&0&0& \ldots & -q_{j-2}& q_{j-1}+b_{j-1} & 0 & 0\\
	 0&0&0& \ldots & 0 & -q_{j-1} & 1+b_{j} & 0\\
	 0&0&0& \ldots & 0 & 0 & 0 & \tilde{b}_{i}
 \end{bmatrix}
 \]
 Clearly the matrix $V$ is invertible (lower triangular, nonzero diagonal elements), although computing
 such an inverse is nontrivial.  However, all we require is the dominant eigenvalue of $WV^{-1}$.  Because  $W$ is fairly sparse
 we will not need to know all the entries of $V^{-1}$ to extract it.  Also, note that $W$ and $V$ are both $(j+1) \times (j+1)$ matrices.

The action of $W$ on $V^{-1}$ is to multiply the last row by $k$, make middle rows 0, and sum the first
 $j$ elements of each columns and multiply it by the constant $\rho = \frac{l\overline{s_v}}{\overline{s}}$. Let
 us denote the elements of the last row of $V^{-1}$ by $\epsilon_k$ and the sums of the first $j$ elements of the
 $k$th column as $\delta_k$.   $WV^{-1}$ then has a relatively simple form
 \[
	 WV^{-1}  = \begin{bmatrix} k\epsilon_1 & \ldots & k\epsilon_{j+1}\\
					0 & \ldots & 0\\
	 				\vdots & & \vdots\\
	 				0 & \ldots & 0\\
	 				\rho \delta_1 & \ldots & \rho \delta_{j+1}
 \end{bmatrix}
 \]
From here we can move ahead with the eigenvalue calculation.  This involves computing the determinant of
\[
	WV^{-1} - \lambda I = \begin{bmatrix} k\epsilon_1 - \lambda & k\epsilon_2 & \ldots & k\epsilon_{j+1}\\
	0 & -\lambda & \ldots & 0\\
	\vdots & \vdots & & \vdots\\
	0 & 0 & \ldots & 0\\
	\rho \delta_1 & \rho \delta_2 & \ldots & \rho \delta_{j+1} - \lambda\\
\end{bmatrix}
\]
Expanding along the first column we have
\begin{multline*}
	\det(WV^{-1} - \lambda I) = (k\epsilon_1-\lambda) \cdot \det\begin{bmatrix}
		-\lambda & 0 & \ldots & 0\\
		\vdots & \vdots& \vdots &\vdots\\
	\rho \delta_2 & \rho \delta_3 & \ldots & \rho \delta_{j+1} - \lambda \end{bmatrix}\\
		+ (-1)^{j+2}\rho \delta_{1}\cdot \det
		\begin{bmatrix} k\epsilon_2 & \ldots & k\epsilon_j & k\epsilon_{j+1}\\
		-\lambda & \ldots & 0 & 0\\
		\vdots &  & \vdots &\vdots\\
	0 & \ldots &-\lambda & 0 \end{bmatrix}
\end{multline*}
Computing the determinants in this expression is straightforward, noting that both are $j \times j$ matrices and the first is a lower triangular
matrix, and thus the determinant is the product of the diagonal elements $(-\lambda)^{j-1}(\rho \delta_{j+1}-\lambda)$.  For the second
we expand along the last column to see that the determinant is
\[
	(-1)^{j+1} k\epsilon_{j+1} \cdot \det(-\lambda I) = (-1)^{j+1} k\epsilon_{j+1}(-\lambda)^{j-1}
\]
Hence
\begin{multline*}
	\det(WV^{-1} - \lambda I) = (k\epsilon_1-\lambda) (-\lambda)^{j-1}(\rho \delta_{j+1}-\lambda)
		\\+ (-1)^{j+2}\rho \delta_{1} (-1)^{j+1} k\epsilon_{j+1}(-\lambda)^{j-1}
\end{multline*}
Hence, the characteristic polynomial only involves $\epsilon_1$, $\epsilon_{j+1}$, $\delta_1$ and $\delta_{j+1}$,  and
we need only know the first and last column of $V^{-1}$.  These can be computed by looking at the first and last row of
the cofactor matrix. For $\epsilon_1$ we look at the minor of $v_{j+1,1}$, and note the top row is all 0's, so that the minor,
and thus the cofactor are $0$ and thus $\epsilon_1 = 0$.  To find $\delta_{j+1}$ we must find the
cofactors of of the first $j$ elements on the bottom row of $V$, but when the bottom row is eliminated to compute the cofactor, the
last column is all 0's, and hence each of the cofactors is 0.  Thus $\delta_{j+1} = 0$.  The characteristic polynomial
then reduces to
\[
	(-\lambda)^{j-1} \lambda^2 - (-\lambda)^{j-1} \rho k \delta_1 \epsilon_{j+1}
\]
Setting this equal to 0 and factoring:
\[
	\lambda = \pm \sqrt{\rho k \delta_1 \epsilon_{j+1}}
\]
We have reduced finding the largest magnitude
eigenvalue to computing the last element of the last column of $V^{-1}$ ($\epsilon_{j+1}$) and the sum of the first $j$ elements in the first
column of $V^{-1}$ ($\delta_1$).  The elements of $V^{-1}$ are not difficult to find, see \cite{guo2006global}.

If we define $q_0=1$ and $q_j=1$ then
\[
\delta_1 = \sum_{k=1}^j \frac{\prod_{\ell=0}^{k-1} q_\ell}{\prod_{\ell=1}^k (q_\ell + b_{\ell})}
 = \sum_{k=1}^j \prod_{\ell=1}^k \frac{q_{\ell+1}}{q_\ell + b_\ell}
\]
We can rewrite this sum as
\[\scalebox{.9}{$\displaystyle{
\delta_1 =\frac{1}{q_1 + b_{1}}\left( 1+ \frac{q_1}{q_2 + b_{2}}\left( 1+  \frac{q_2}{q_3 + b_{3}}\left( 1+ \ldots
\frac{q_{j-2}}{q_{j-1} + b_{j-1}}\left( 1 + \frac{q_{j-1}}{1 + b_{j}} \right) \ldots \right)\right)\right)}$}
\]
This gives that
\[\scalebox{.85}{$\displaystyle{
	R_0 = \sqrt{\frac{\rho k}{\tilde{b}_{i}} \frac{1}{q_1 + b_{1}}\left( 1+ \frac{q_1}{q_2 + b_{2}}\left( 1+  \frac{q_2}{q_3 + b_{3}}\left( 1+ \ldots
\frac{q_{j-2}}{q_{j-1} + b_{j-1}}\left( 1 + \frac{q_{j-1}}{1 + b_{j}} \right) \ldots \right)\right)\right)}}$}
\]
Recall $\rho=\frac{l\overline{s_v}}{\overline{s}}$.  Moving out of dimensionless form, we find
\[\scalebox{.79}{$\displaystyle{
			R_0 = f\sqrt{\frac{cc_v \overline{S}_v}{\tilde{\mu}\overline{S}} \frac{1}{\alpha_1 + \mu_{1}}\left( 1+ \frac{\alpha_1}{\alpha_2 + \mu_{2}}
					\left( 1+  \frac{\alpha_2}{\alpha_3 + \mu_{3}}\left( 1+ \ldots
	\frac{\alpha_{j-2}}{\alpha_{j-1} + \mu_{j-1}}\left( 1 + \frac{\alpha_{j-1}}{\gamma + \mu_{j}} \right) \ldots \right)\right)\right)}}$}
\]
This is the form that was conjectured in \cite{johnsonPLOSNTD}.  Alternatively,
\[
R_0= f \sqrt{\frac{cc_v\overline{S}_v}{\tilde{\mu}\overline{S}} \sum_{k=1}^j \prod_{l=1}^k \frac{\alpha_{l-1}}{\alpha_l + \mu_l} }
\]
where $\alpha_0=1$ and $\alpha_j=\gamma$.  This form best displays the dependence of $R_0$ on the number of relapses, $j-1$.  This is our first step in
quantifying how the number of relapses affect the spread of a vector-borne disease.

Notice that $R_0$ is directly proportional to the biting rate $f$, the roots of the competencies $c$ and $c_v$, and the
root of the carrying capacity of the vector population.  $R_0$ is inversely proportional to the death rate of the vectors
$\tilde{\mu}$ and the carrying capacity of the host population $\overline{S}$.  Thus $R_0$ can be completely controlled
by these parameters.  For instance, if the biting rate or the competencies are set to 0, the disease will eventually be elimnated from the population.
Similarly, the higher the death rate of the vectors, the smaller $R_0$ will be, inhibiting the spread of the disease.
It is also worth noting that $R_0$ only depends on the ratio of the vector and host populations
sizes only.  If the ratio of the populations is the same, $R_0$ is the same.  

To investigate how $R_0$ depends on the transfer rates $\alpha_i$, consider a model with  
$\mu_i = 0$.  Then, letting $\beta = f \sqrt{\frac{cc_v\overline{S}_v}{\mu\overline{S}}}$, we have
\[
	R_0^2 = \beta^2 \left( \frac{1}{\alpha_1} \left( 1 + \frac{\alpha_1}{\alpha_2} \left( 1 + \frac{\alpha_2}
				{\alpha_3}\left( 1 +\ldots \frac{\alpha_{j-2}}{\alpha_{j-1}}\left( 1+  \frac{\alpha_{j-1}}{
		\alpha_j}\right)\ldots \right)\right) \right) \right)
\]
\[
	 = \beta^2 \left( \frac{1}{\alpha_1} + \frac{1}{\alpha_2} \left( 1 + \frac{\alpha_2}
			 {\alpha_3}\left( 1 +\ldots \frac{\alpha_{j-2}}{\alpha_{j-1}}\left( 1+  \frac{\alpha_{j-1}}{
		\alpha_j}\right)\ldots \right)\right) \right) 
\]
\[
	= \beta^2 \left( \frac{1}{\alpha_1} + \frac{1}{\alpha_2}  + \frac{1}
		{\alpha_3}\left( 1 +\ldots \frac{\alpha_{j-2}}{\alpha_{j-1}}\left( 1+  \frac{\alpha_{j-1}}{
		\alpha_j}\right)\ldots \right)\right)  
\]
and so on, until
\[
	R_0^2 = \beta^2 \left( \frac{1}{\alpha_1} + \frac{1}{\alpha_2} + \frac{1}{\alpha_3} + \ldots + \frac{1}{\alpha_{j-1}}
			 + \frac{1}{\alpha_j} \right)
\]
Let $T_i$ be the average amount of time spent in the $i$th compartment, then $T_i \propto \frac{1}{\alpha_i}$, so that
\[
	R_0^2  \propto \beta^2 \sum T_i.
\]
As the rates increase, the $T_i$ will decrease and thus $R_0$ will decrease as well.  Also, we can add relapses
to the model but keep $R_0$ the same by fixing the total amount of time spent in the relapsing states.  Thus, we can conclude that the addition
of relapses to the model increases $R_0$ by increasing the amount of time that an infected host spends
being infectious.

\section{The Bifurcation at $R_0 = 1$.}
\label{bifsec}
We have already shown that there exists a bifurcation in the vector-borne relapsing disease model by finding a finite $R_0$.  We wish to learn more about the bifurcation.
To do this we shall take advantage of center manifold theory.  An introduction to such theory is written elsewhere \cite{wiggins2003introduction, carrcenter}
, but for our purposes it is enough to note that
if the zero eigenvalue of the linear term is simple, then the dynamics on the center manifold are one dimensional.
It is clear then that bifurcations with one simple zero eigenvalue are much simpler to study.  We use this fact to investigate the stability
of the EE near $R_0 = 1$.  To do this, we insert a parameter $\mu$ into the equations where
\[
	\mu = 0 \iff R_0=1.
\]
Define $\mu = R_0-1$ and so, solving for the biting rate:
\[
	f = \frac{\mu+1}{\sqrt{\frac{cc_v \overline{S}_v}{\tilde{\mu}\overline{S}} \frac{1}{\alpha_1 + \mu_{1}}\left( 1+ \frac{\alpha_1}{\alpha_2 + \mu_{2}}
					\left( 1+  \frac{\alpha_2}{\alpha_3 + \mu_{3}}\left( 1+ \ldots
	\frac{\alpha_{j-2}}{\alpha_{j-1} + \mu_{j-1}}\left( 1 + \frac{\alpha_{j-1}}{\gamma + \mu_{j}} \right) \ldots \right)\right)\right)}}.
\]
All other parameters  being constant, we will abbreviate this as $f = \frac{\mu+1}{\zeta}$.
The Jacobian of the system at the DFE can be written in block form
\[
\begin{pmatrix}F-V & 0\\ -J_3 & -J_4 \end{pmatrix}.
\]
Let $v$ and $w$ be the left and right eigenvectors corresponding to the eigenvalue 0.  Without loss of generality we can choose these such that $vw=1$.
In the proof of Lemma 3 in \cite{compart} we see that we can also say that $v_1, w_i \geq 0$ for $1 \leq i \leq m$.
Define
\[
	a = \sum_{i,j,k = 1}^m v_iw_jw_k \left( \frac{1}{2} \frac{\partial^2 f_i}{\partial x_j \partial x_k}(x_0,0) + \sum_{l=m+1}^n \varepsilon_{lk}
	\frac{\partial^2 f_i}{\partial x_j \partial x_l}(x_0,0) \right),
\]
\[
	b = \sum_{i,j=1}^n v_iw_j \frac{\partial^2 f_i}{\partial x_j \partial\mu}(x_0,0),
\]
where $\varepsilon_{lk}$, $l=m+1, \ldots, n$, $k=1, \ldots, m$ are the $(l-m,k)$ entries of $-J_4^{-1}J_3$, when $R_0 = 1$.
The following theorem is found in  \cite{compart}
\begin{theorem}
	\label{transcb}
	In a disease transmission model satisfying  conditions 1-5,
	with the parameter $\mu$ as described above, with zero as a simple eigenvalue, and $b\neq 0$,
there is a $\delta > 0$ such that
\begin{itemize}
	\item if $a<0$ then there are locally asymptotically stable endemic equilibria near $x_0$ for $0 < \mu < \delta$
	\item if $a>0$, then there are unstable endemic equilibria near $x_0$ for $-\delta < \mu <0$.
\end{itemize}
\end{theorem}
The proof that $0$ is a simple eigenvalue is found in Appendix \ref{simpproof} and the proof for the following Lemma is found in Appendix \ref{lem1proof}.
\begin{lemma}
	$b \neq 0$, $a>0$.
\end{lemma}
We can then apply Theorem 1 of \cite{compart} to our system, which states
that there are asymptotically stable equilibria near the DFE when $R_0$ is sufficiently close to, but greater than 1.
\section{Existence of Endemic Equilibria for all $R_0$}
Having established the existence of branch of stable endemic equilibria (EE) near the bifurcation,
there is a natural question regarding the behavior of these EE outside of the
neighborhood of the bifurcation.  Specifically, how does the local branch found in the previous section extends for larger (or smaller) values of $R_0$.
The complex form of $R_0$ and the arbitrary number of equations appear to make this problem quite difficult.
But the majority of the equations have a simple linear form, and from these we can derive a simple recurrence relation for nontrivial equilibrium
values of the $I_2$ through $I_{j+1}$ in terms of $I_1$.  What remains is a fixed number of equations to solve.  We thus make the following proposition
\begin{proposition}
	Given any $I_1 \geq0$ there exist unique values $S, \widetilde{S}, R, I_2, \ldots, I_j,$ and $\widetilde{I}$ such that
	\[
		S'=\widetilde{S}' = R' = I_2' = \ldots = I_j' = \widetilde{I}' = 0.
	\]
\end{proposition}
\begin{proof}
	Let $I_1$ be fixed.  Consulting the equations for $I_2$ through $I_{j-1}$ we see that
\[
	I_k' = 0 \iff I_k = \frac{\alpha_{k-1}I_{k-1}}{\alpha_k+\mu_k} = c_{k-1}I_{k-1} \mbox{ for } 2 \leq k \leq j-1.
\]
For $I_j$ 
\[
	I_j'=0 \iff I_j = \frac{\alpha_{j-1}I_{j-1}}{\gamma+\mu_j}  =  c_{j-1}I_{j-1},
\]
and
\[
	R'=0 \iff R = \frac{\gamma I_j}{\mu_r} = c_jI_j.
\]
Then for $I_2, \ldots,I_j, R$ there is a simple multiplicative recurrence relation which is solved easily for $I_k$ and $R$, namely
\[
	I_k = c_{k-1} \ldots c_1I_1;\; \quad R = c_k \ldots c_1 I_1.
\]
Thus the unique steady states for $I_2, \ldots I_j$ and $R$ are determined uniquely by $I_1$. Now observe that at a steady state, inserting $c_0=1$,
\[
	\sum_{k=1}^j I_k = \sum_{k=1}^j c_{k-1}\ldots c_0I_1.
\]
Let $\xi = \sum_{k=1}^j c_{k-1}\ldots c_0$ and $\sum_{k=1}^j I_k = \xi I_1$.
Then we have 
\[
	\widetilde{S}' = 0 \iff \widetilde{\mu}_s(\overline{S}_v - \widetilde{S}) - \frac{fc\widetilde{S}}{\overline{S}}\xi I_1 = 0,
\]
and thus
\[
	\widetilde{S} = \frac{\widetilde{\mu}_s\overline{S}_v}{\widetilde{\mu}_s + \frac{fcI_1}{\overline{S}} \xi}.
\]
So then $\widetilde{S}$ is uniquely determined by $I_1$.  Now, since $\widetilde{I}  =\overline{S}_v - \widetilde{S} $, we should confirm that this value gives $\widetilde{I}'=0$ as follows.
\[
	\widetilde{I}' = \frac{fc\widetilde{S}I_1 \xi}{\overline{S}} - \widetilde{\mu}(\overline{S}_v - \widetilde{S})
	= \widetilde{S} \left(\frac{fcI_1\xi}{\overline{S}} + \widetilde{\mu} \right) - \widetilde{\mu}_s \overline{S}_v,
\]
but since $\widetilde{\mu} = \widetilde{\mu}_s$, and given the equilibrium value for $\widetilde{S}$ we have
\[
	\widetilde{S} \left(\frac{fcI_1\xi}{\overline{S}} + \widetilde{\mu}_s \right) - \widetilde{\mu}_s \overline{S}_v = \widetilde{\mu}_s \overline{S}_v -\widetilde{\mu}_s \overline{S}_v
	 = 0,
\]
and $I_1$ uniquely determines the equilibrium value for $\widetilde{I}$.
Lastly we have 
\[
	S' = 0 \iff S = \frac{\mu_s\overline{S}}{\mu_s + \frac{fc_v\widetilde{I}}{\overline{S}}},
\]
and as $\widetilde{I}$ is uniquely determined by $I_1$, so is $S$.
\end{proof}
The consequence of this proposition is that the number of equilibrium points is determined by the number of values of $I_1$ such that $I_1' = 0$.  Furthermore notice that  $I_1 = 0$
implies that $S=\widetilde{S} = R = I_2 = \ldots = I_j = \widetilde{I} = 0$, and we have generated the DFE. So we will only be looking for values such that $I_1 > 0$.  We will also want
to see how these values depend on $R_0$, so it is useful to note the following form
\[
	R_0 = f \sqrt{\frac{cc_v\overline{S}_v}{\overline{S}\widetilde{\mu}} \frac{1}{\alpha_1+\mu_1} \xi}.
\]
Now, $I_1'=0$ if and only if
\[
	I_1(\alpha_1 + \mu_1) = \frac{fc_v\widetilde{I}S}{\overline{S}}.
\]
But since we also have $S'=0$ this leads to
\[
	\frac{fc_v\widetilde{I}S}{\overline{S}} = \mu_s(\overline{S} - S).
\]
Thus
\[
	I_1 = \frac{\mu_s(\overline{S} - S)}{\alpha_1 + \mu_1} = \frac{\mu_s\overline{S}}{\alpha_1 + \mu_1} \left( 1-\frac{\mu_s}{\mu_s + \frac{fc_v\widetilde{I}}{\overline{S}}} \right)
	= \frac{\mu_s\overline{S}}{\alpha_1 + \mu_1} \left(\frac{fc_v\widetilde{I}}{\mu_s\overline{S} + fc_v\widetilde{I}}\right).
\]
Now, at an equilibria, we can write $\widetilde{I}$ uniquely in terms of $I_1$:
\[
	\widetilde{I} = \overline{S}_v - \widetilde{S} = \overline{S}_v - \frac{\widetilde{\mu}_s\overline{S}_v}{\widetilde{\mu}_s + \frac{fcI_1}{\overline{S}} \xi}
	= \overline{S}_v\left( 1 - \frac{\widetilde{\mu}_s}{\widetilde{\mu}_s + \frac{fcI_1}{\overline{S}} \xi} \right)
	 =  \frac{\overline{S}_vfcI_1\xi}{\overline{S}\widetilde{\mu}_s + fcI_1\xi}.
\]
Substituting this back into the expression for $I_1$ gives
\[
	I_1 = \frac{\mu_s\overline{S}}{\alpha_1 + \mu_1} \left(\frac{\overline{S}_vf^2cc_vI_1\xi}{\overline{S}\mu_s(\overline{S}\widetilde{\mu}_s + fcI_1\xi) + \overline{S}_vf^2cc_vI_1\xi} \right).
\]
Because we are interested in solutions where $I_1\neq 0$, we can divide both sides of the equation by $I_1$ and rearrange
\[
	\overline{S}^2\mu_s\widetilde{\mu}_s + fc\overline{S}\mu_s\xi I_1 + \overline{S}_vf^2cc_v\xi I_1 = \frac{\mu_s\overline{S}\,\overline{S}_vf^2cc_v\xi}{\alpha_1+\mu_1}.
\]
Thus
\[
	I_1 = \frac{1}{fc\overline{S}\mu_s\xi + \overline{S}_vf^2cc_v\xi} \left(\frac{\mu_s\overline{S}\,\overline{S}_vf^2cc_v\xi}{\alpha_1+\mu_1}  - \overline{S}^2\mu_s\widetilde{\mu}_s\right)
\]
\[
	= \frac{\overline{S}^2\mu_s\widetilde{\mu_s}}{fc\overline{S}\mu_s\xi + \overline{S}_vf^2cc_v\xi}\left(
	\frac{\overline{S}_vf^2cc_v\xi}{\widetilde{\mu}_s\overline{S}(\alpha_1+\mu_1)}  - 1\right),
\]
and since $\widetilde{\mu}_s = \widetilde{\mu}$ this can be rewritten
\[
	I_1 = \frac{\overline{S}^2\mu_s\widetilde{\mu_s}(R_0^2-1)}{fc\overline{S}\mu_s\xi + \overline{S}_vf^2cc_v\xi}.
\]
Thus we only have a nonzero equilibrium when $R_0 \neq 1$, and when $R_0<1$ the nonzero equilibrium is negative, and when $R_0 >1$ there is one endemic equilibrium.  The content of the above
discussion is contained in the following theorem:
\begin{theorem}
	Given a simple vector-borne relapsing disease model \ref{host}, \ref{vec}, for every value of $R_0>1$ there is exactly one nonzero EE.  For $R_0$ sufficiently close to 1, this EE is locally asymptotically stable.
\end{theorem}
The form of the EE is complicated so evaluating the Jacobian at the EE is a difficult problem, and thus so is determining their stability away from $R_0=1$. 
The most efficient way to determine the stability of the EE is with Lyapunov
functions (LF), though the standard LF for vector-borne diseases with no relapses in \cite{timescale} did not yield any results. 
The LF in \cite{guo2006global} could potentially be extended to this case.

\section{Removal to the Recovered Compartment}
In TBRF there can be some variation in the number of relapses experienced by infected hosts.  We introduce this variation
into our model by allowing individuals to leave an infected compartment and go directly to the recovered state.  This can also correspond to treatment of the
disease at that particular compartment.
Let $\theta_i \geq 0$ be the transfer rate to the recovered compartment out of the $i$th infected compartment.  
The equations
change only slightly: the $\mu_i$ get replaced with $\mu_i + \theta_i$ and the recovered equation changes.
\begin{equation}
	\begin{array}{ccl}
		S' &=& \beta(N) - fc_v\widetilde{I} \frac{S}{N} - \mu_sS,\\
		I'_1 &=& fc_v\widetilde{I} \frac{S}{N} - \alpha_1I_1 - (\mu_1 + \theta_1)I_1,\\
		I'_2 &=& \alpha_1I_1-\alpha_2I_2 - (\mu_2 +\theta_2)I_2,\\
		 & \vdots& \\
		I'_{j-1} &=& \alpha_{j-2}I_{j-2} - \alpha_{j-1}I_{j-1} - (\mu_{j-1} + \theta_{j-1})I_{j-1},\\
		I'_j &=& \alpha_{j-1}I_{j-1} - \gamma I_j - \mu_jI_j,\\
		R' &=& \sum_{i=1}^{j-1} \theta_iI_i + \gamma I_j - \mu_rR.
	\end{array}
	\label{hostre}
\end{equation}
The vector equations remain unchanged:
\begin{equation}
	\begin{array}{ccl}
		\widetilde{S}' &=& \beta_v(\widetilde{N},N) - \frac{fc\widetilde{S}}{N} \sum_{k=1}^j I_k - \widetilde{\mu}_s\widetilde{S},\\
		\widetilde{I}' &=& \frac{fc\widetilde{S}}{N} \sum_{k=1}^j I_k -\widetilde{\mu}\widetilde{I}. 
	\end{array}
	\label{vecre}
\end{equation}
Under the assumption that the population is constant ($N=\overline{S}$) we sum these equations to get
\[
	0 = N'=\beta(\overline{S}) -\sum_{i=1}^{j-1} \mu_iI_i - \gamma I_j - \mu_rR,
\]
and following the proof of Proposition \ref{eqdeath} in the Appendix we get that all the $\mu_i = \mu_r=\mu_s$.  Similarly,
if we assume $\mu_i = \mu_r=\mu_s$, then we have that $N=\overline{S}$ is invariant.

It is straightforward to show that this system satisfies the necessary conditions for the use of the next generation method. 
\begin{proposition} The system \eqref{hostre}, \eqref{vecre} satisfies Conditions 1-5.
\end{proposition}
The proof is done in the exact same way as in Appendix \ref{welldef}.

Now we move ahead to compute $R_0$.  However, note that this process involves only the infected equations and not the 
recovered equation.  So the process is exactly the same as in the last section, but with $\mu_i$ replaced by $\mu_i
+\theta_i$.  Thus when we have removal to the recovered compartment:
\[\scalebox{.79}{$\displaystyle{
			R_0 = f\sqrt{\frac{cc_v \overline{S}_v}{\tilde{\mu}\overline{S}} \frac{1}{\alpha_1 + \mu_{1}+\theta_1}
				\left( 1+ \frac{\alpha_1}{\alpha_2 + \mu_{2}+\theta_2}
					\left( 1+  \frac{\alpha_2}{\alpha_3 + \mu_{3}+\theta_3}\left( 1+ \ldots
	\frac{\alpha_{j-2}}{\alpha_{j-1} + \mu_{j-1}+\theta_{j-1}}\left( 1 + \frac{\alpha_{j-1}}{\gamma + \mu_{j}} \right) 
	\ldots \right)\right)\right)}}$}.
\]
Alternatively,
\begin{equation}
	\label{rnotrem}
R_0= f \sqrt{\frac{cc_v\overline{S}_v}{\tilde{\mu}\overline{S}} \sum_{k=1}^j \prod_{l=1}^k \frac{\alpha_{l-1}}{\alpha_l + \mu_l +\theta_l} }
\end{equation}
where $\alpha_0=1$ and $\alpha_j=\gamma$.
As one might expect, removal to the recovered compartment drives down $R_0$ since the $\theta_i \geq 0$ only appear in the
denominators.  The magnitude of this contribution is determined by the size of $\theta_i$ compared
to $\alpha_{i-1}$.  If the recovery rate from the $i$th compartment is small compared to the rate at which individuals
are being transferred into that compartment, then it has little effect on the spread of the disease.  Conversely, if 
the direct recovery from the $i$th infected compartment is large compared to the rate at which individuals are transferred
in, it will result in a more significant mitigation of the disease spread.

We can also use this result to determine at which stage treatment of the disease is most effective, in terms of reducing $R_0$.
Treatment at the $\i$th compartment is equivalent to increasing $\theta_i$, which then reduces the size of the terms in the sum
in Equation \eqref{rnotrem} for $k\geq i$.  Hence, the compartment that reduces the most terms yields the largest decrease in $R_0$ i.e. the first compartment.
This matches our intuition, since we most drastically reduce total infection time, and thus the ability of the host to infect a susceptible vector, by treating the disease before the first relapse.
\subsection{The Bifurcation at $R_0=1$}
There are still questions about which results from the previous section are easily extended to the case with removal
to the recovered compartment, e.g., the transcritical bifurcation at $R_0 = 1$ and the number of EE.  As an opening
step we consider the Jacobian matrix for this system:
\[\scalebox{.75}{$\displaystyle{
	\begin{pmatrix} \alpha_1 + \mu_1+\theta_1 & 0 & 0 & \ldots & 0 & 0 & -fc_v & 0 & 0 & 0 \\
	-\alpha_1 & \alpha_2+\mu_2+\theta_2 & 0 & \ldots & 0 & 0 & 0 & 0 & 0 & 0\\
	0 & -\alpha_2 & \alpha_3+\mu_3 + \theta_3 & \ldots  & 0 & 0 & 0 & 0 & 0 & 0\\
	\vdots &\vdots &\vdots & \ddots & \vdots &\vdots &\vdots &\vdots &\vdots &\vdots\\ 
	0 & 0 & 0& \ldots & -\alpha_{j-1} & \gamma + \mu_j& 0 & 0 & 0 & 0\\
	-\frac{fc\bar{S}_v}{\bar{S}} &	-\frac{fc\bar{S}_v}{\bar{S}} &	-\frac{fc\bar{S}_v}{\bar{S}} & \ldots &-\frac{fc\bar{S}_v}{\bar{S}}&-\frac{fc\bar{S}_v}{\bar{S}} 
	&\tilde{\mu} &0&0&0\\
	0&0&0& \ldots &0&0& fc_v & \mu_s& 0 &0\\
	\frac{fc\tilde{S}}{N} &\frac{fc\tilde{S}}{N} &\frac{fc\tilde{S}}{N} & \ldots &\frac{fc\tilde{S}}{N} &\frac{fc\tilde{S}}{N} &0&0&\tilde{\mu}_s&0\\
			-\theta_1&-\theta_2&-\theta_3& \ldots & -\theta_{j-1} & -\gamma &0&0&0&\mu_r
\end{pmatrix}}$}
\]
To apply Theorem \ref{transcb} we need to know the multiplicity for the 0 eigenvalue for this matrix.
The simplicity of the 0 eigenvalue and the proof that $b \neq 0$ (as defined in the previous sections) are precisely the same here as in Appendix \ref{simpproof}.
While we now have more nonzero elements in $-J_4^{-1}J_3$, the $\epsilon_{3,k}$, are cancelled out
since the associated second derivatives in the equation for $a$ are the derivative of an infected variable with respect to the
recovered variable, and thus are 0.  Thus, the computation for $a$ is exactly the same as Appendix \ref{lem1proof} and we have $a<0$.
Hence, we have the following Corollary:
\begin{corollary}
	The nontrivial DFE of the system \eqref{hostre}, \eqref{vecre} undergoes a transcritical bifurcation as $R_0$ goes
	above 1, and has a branch of locally asymptotically stable EE for $R_0$ sufficiently close to 1.
\end{corollary}
\section{Discussion and Future Work}
Using the next generation method and standard matrix computation methods we have found a form for $R_0$ of
vector-borne relapsing diseases with an arbitrary number of relapses.  From this we conclude that $R_0$ increases
as the number of relapses of the disease increases, with all other parameters fixed. 
We have also taken advantage of results in \cite{compart} to show the existence of a branch of endemic equilibria that are locally asymptotically stable for $R_0$
sufficiently close to 1.  A straightforward calculation demonstrated that only one such branch of endemic
equilibria can exist.  Both of these results are independent of the number of relapses.  The form of $R_0$ did not yield any particularly enlightening control strategies
for the disease: $R_0$ can be completely controlled through, biting rate, competencies, population size, and vector death rates.

The computation of $R_0$ relied on the assumption of a constant population in the hosts.  This assumption was
shown to be equivalent to equal death rates in the infected host compartments.
  Allowing for variable death rates among the compartments changes the form of the
Jacobians that make up the next
generation matrix, though future work may show that the computations for $R_0$ are similar.  However, in the case
of equal death rates the constant population is found to be attracting and thus we need only
study the dynamics restricted to this constant population.

Future work will consider refinements of our model which incorporate a period of non-infectivity (a latent state)
in the hosts between
relapses.
The relapses of TBRF are driven by antigenic variation of \textit{Borellia} spirochetes
within the host \cite{antigen}.
Preceding this change is the host's immune response which nearly eradicates the bacteria from the
host.  However, \textit{Borellia} can initiate a full infection with a single spirochete.  Further
details on the relapse mechanism can be found in \cite{antigen}.  The apparent lack of spirochetes in the host
results in a week of apparent health between relapses of TBRF and any susceptible tick that bites the host
will not become infected. Hence, hosts in this
latent state will not drive the infection of susceptible vectors.  Further work is focusing on the
changes in $R_0$ and the dynamics caused by the
addition of latent states.

\bibliographystyle{apalike}      
\bibliography{currentresearch}   

%
%
\appendix
\section{Appendix}
\subsection{The Theory of Compartmental Models}
\label{welldef}
We give a brief description of compartmental models here, following their development in
\cite{compart}.  Consider a population which can be separated into $n$ homogeneous compartments, with the number of members in each compartment
represented by the vector $\bm{x} \in \mathbb{R}^n$.  We suppose that the first $m$ compartments represent infected states while the remaining $n-m$ compartments
are uninfected states.  It is natural to insist that $\bm{x} \geq \bm{0}$ (inequality taken componentwise) since the $x_i$ represent populations.
Let $\bm{X}_s = \{\bm{x} \geq \bm{0}:
x_i = 0, i=1, \ldots m\}$ be the disease free states.  Let $\mathcal{F}_i(\bm{x})$ be the number of new infections in compartment $i$ (autonomy is assumed).  $\mathcal{V}_i^+(\bm{x})$ is
the rate of transfer of individuals into compartment $i$ and $\mathcal{V}_i^-(\bm{x})$ is the rate of transfer out.  Assume that these function are at least twice continuously differentiable
and we can write our disease transmission model as
\begin{equation} \label{compsys}
	\dot{x}_i = f_i(\bm{x}) = \mathcal{F}_i(\bm{x}) +\mathcal{V}_i^+(\bm{x}) - \mathcal{V}_i^-(\bm{x}) \qquad i=1,\ldots n
\end{equation}
Five conditions need to be met in order to split the system in
such a way that the computation of $R_0$ is straightforward.  They are as follows.
\begin{condition}{1}
\bm{x} \geq 0 \Rightarrow \mathcal{F}_i, \mathcal{V}_i^+, \mathcal{V}_i^- \geq 0
\end{condition}
This is natural to assume as these functions represent a transfer of individuals between compartments.
\begin{condition}{2}
	x_i=0 \Rightarrow \mathcal{V}_i^-(\bm{x}) = 0
\end{condition}
This condition requires that no individuals can transfer out of an empty compartment.  In particular, if we are in $\bm{X}_s$ then we have $\mathcal{V}_i = 0$ for $i=1,2,\ldots m$.
These two conditions are enough to prove that solutions of the ODE will remain nonnegative when the initial conditions are nonnegative.
\begin{condition}{3}
	\mathcal{F}_i = 0  \mbox{ for } i>m
\end{condition}
Hence there are no new infections in the noninfected compartments.
\begin{condition}{4}
	\bm{x} \in \bm{X}_s \Rightarrow \mathcal{F}_i(\bm{x}) = 0 \mbox{ and } \mathcal{V}_i^+(\bm{x}) = 0 \mbox{ for } i=1,\ldots,m
\end{condition}
This condition insures that in disease free states there are no new infections in the infected compartments and there are no individuals being transferred into those compartments.
Now assume that $\bm{x_0} \in \bm{X}_s$ is a fixed point of \eqref{compsys}.  Such points are called Disease Free Equilibria (DFE).  We consider only systems
where DFE are stable in the absence of infection.  Let  $\mathcal{F}(\bm{x}) = (\mathcal{F}_1(\bm{x}), \ldots, \mathcal{F}_n(\bm{x}))^T$ then
\begin{condition}{5}
	\mathcal{F}(\bm{x}) \equiv \bm{0} \Rightarrow \mbox{ The DFE is stable.}
\end{condition}
In particular this implies that the Jacobian $Df(\bm{x}_0)$ has eigenvalues with negative real parts.
Let $\mathcal{V}_i = \mathcal{V}_i^--\mathcal{V}_i^+$ and $\mathcal{V} = (\mathcal{V}_1, \ldots, \mathcal{V}_n)^T$.  Then $f=\mathcal{F} - \mathcal{V}$.  It can be shown that
\[
	Df(\bm{x}_0) = D\mathcal{F}(\bm{x}_0) - D\mathcal{V}(\bm{x}_0) =
\begin{pmatrix} F & 0\\ 0 & 0 \end{pmatrix} - \begin{pmatrix} V & 0 \\ J_3 & J_4 \end{pmatrix}
\]
The $n \times n$ matrix $F$ is nonnegative, the $n \times n$ matrix $V$ is nonsingular and $J_4$ has eigenvalues with positive real parts.  The matrix $FV^{-1}$ is
called the next generation matrix.  Theorem 2 of \cite{compart} says that given a system \eqref{compsys} satisfying Conditions 1-5, if $\rho(FV^{-1})<1$ the DFE
is asymptotically stable, and if $\rho(FV^{-1}) > 1$ the DFE is unstable.  We confirm now that our model meets the conditions.
\begin{proposition}
	The system \eqref{host}, \eqref{vec}  satisfies conditions 1-5.
\end{proposition}
\begin{proof}
Rearrange the system into a vector with the infected hosts, infected vectors, susceptible host recovered host, and susceptible vector :
\begin{equation}
\frac{d}{dt}\begin{pmatrix} I_1\\ I_2\\ \vdots\\ I_{j-1}\\I_j\\ \widetilde{I}\\ S\\R\\\widetilde{S} \end{pmatrix}
= \begin{pmatrix}fc_v\widetilde{I}\frac{S}{N}\\ 0 \\ \vdots \\ 0 \\ 0 \\ \frac{fc\widetilde{S}}{N}\sum_{k=1}^j I_k\\ 0\\0\\0 \end{pmatrix}
+ \begin{pmatrix} 0 \\ \alpha_1I_1 \\ \vdots \\ \alpha_{j-2}I_{j-2}\\ \alpha_{j-1}I_{j-1}\\ 0\\ \beta(N)\\\beta(N, \widetilde{N}) \\\gamma I_j\end{pmatrix}
-\begin{pmatrix} (\alpha_1 + \mu_1)I_1\\ (\alpha_2 + \mu_2)I_2\\ \vdots \\ (\alpha_{j-1} + \mu_{j-2})I_{j-2}\\(\alpha_{j-1} + \mu_{j-1})I_{j-1}
\\(\gamma + \mu_j)I_j\\ \widetilde{\mu}\widetilde{I}\\ fc_v\widetilde{I}\frac{S}{N} + \mu_sS\\ \frac{fc\widetilde{S}}{N}\sum_{k=1}^j I_k +
\widetilde{\mu}_s\widetilde{S}\\ \mu_rR\end{pmatrix}
\label{sysvecform}
\end{equation}
Once we have written the system in the form \eqref{compsys}, conditions 1,2 and 3 follow directly from this form.
Setting $I_1 = \ldots = I_j = \widetilde{I} = 0$ gives condition 4. For the final condition, assume that the first vector on the right hand side of the
above equation is $\bm{0}$.  Then, taking the Jacobian will yield a lower triangular matrix, because each equation does not involve a variable beyond its row.
Furthermore, the elements of the diagonal of this matrix, when evaluated that a DFE $(0,\ldots, 0, \overline{S}, \overline{R}, \overline{S}_v)$,  are
\[
	-(\alpha_1 + \mu_1), \ldots, -(\alpha_{j-1} + \mu_{j-1}), -(\gamma + \mu_j), -\widetilde{\mu}, -\mu_s, -\mu_r, -\widetilde{\mu}_s
\]
and as all the parameters are held to be positive, the diagonal elements, which are also the eigenvalues, are negative.
\end{proof}

\subsection{Simplicity of the 0 eigenvalue.}
\label{simpproof}
Consider the system \eqref{compsys}.
It is an easy exercise to show that the Jacobian of \eqref{compsys} takes the
form
\[
	\begin{pmatrix} -\alpha_1 - \mu_1 & 0 & 0 & \ldots & 0 & 0 & \frac{fc_vS}{N} & 0 & 0 & 0 \\
	\alpha_1 & -\alpha_2-\mu_2 & 0 & \ldots & 0 & 0 & 0 & 0 & 0 & 0\\
	0 & \alpha_2 & -\alpha_3-\mu_3 & \ldots  & 0 & 0 & 0 & 0 & 0 & 0\\
	\vdots &\vdots &\vdots & \ddots & \vdots &\vdots &\vdots &\vdots &\vdots &\vdots\\
	0 & 0 & 0& \ldots & \alpha_{j-1} & -\gamma - \mu_j& 0 & 0 & 0 & 0\\
	\frac{fc\tilde{S}}{N} &	\frac{fc\tilde{S}}{N} &	\frac{fc\tilde{S}}{N} & \ldots &\frac{fc\tilde{S}}{N} &	 \frac{fc\tilde{S}}{N} & -\tilde{\mu} &0&0&0\\
	0&0&0& \ldots &0&0& -\frac{fc_vS}{N} & -\mu_s& 0 &0\\
	-\frac{fc\tilde{S}}{N} &-\frac{fc\tilde{S}}{N} &-\frac{fc\tilde{S}}{N} & \ldots &-\frac{fc\tilde{S}}{N} &-\frac{fc\tilde{S}}{N} &0&0&-\tilde{\mu}_s&0\\
	0&0&0& \ldots & 0 & \gamma &0&0&0&-\mu_r
	\end{pmatrix}
\]
We evaluate this at the DFE and
then determine the algebraic multiplicity of the zero eigenvalue when $R_0 = 1$.  The eigenvalue matrix is
\[\scalebox{.85}{$\displaystyle{
	\begin{pmatrix} \lambda+\alpha_1 + \mu_1 & 0 & 0 & \ldots & 0 & 0 & -fc_v & 0 & 0 & 0 \\
	-\alpha_1 & \lambda+\alpha_2+\mu_2 & 0 & \ldots & 0 & 0 & 0 & 0 & 0 & 0\\
	0 & -\alpha_2 & \lambda+\alpha_3+\mu_3 & \ldots  & 0 & 0 & 0 & 0 & 0 & 0\\
	\vdots &\vdots &\vdots & \ddots & \vdots &\vdots &\vdots &\vdots &\vdots &\vdots\\
	0 & 0 & 0& \ldots & -\alpha_{j-1} & \lambda + \gamma + \mu_j& 0 & 0 & 0 & 0\\
	-\frac{fc\bar{S}_v}{\bar{S}} &	-\frac{fc\bar{S}_v}{\bar{S}} &	-\frac{fc\bar{S}_v}{\bar{S}} & \ldots &-\frac{fc\bar{S}_v}{\bar{S}}&-\frac{fc\bar{S}_v}{\bar{S}}
	&\lambda  + \tilde{\mu} &0&0&0\\
	0&0&0& \ldots &0&0& fc_v & \lambda+\mu_s& 0 &0\\
	\frac{fc\tilde{S}}{N} &\frac{fc\tilde{S}}{N} &\frac{fc\tilde{S}}{N} & \ldots &\frac{fc\tilde{S}}{N} &\frac{fc\tilde{S}}{N} &0&0&\lambda+\tilde{\mu}_s&0\\
	0&0&0& \ldots & 0 & -\gamma &0&0&0&\lambda+\mu_r
\end{pmatrix}}$}
\]
The Jacobian matrix has block form
\[
	Df(\bm{x}_0) = \begin{pmatrix}
		T&0\\ L& D
	\end{pmatrix}
\]
 and the multiplicity of the zero eigenvalue will be the sum of the multiplicities of the diagonal blocks \cite{matmat}.
However, the multiplicity of 0 in $D$ is 0, since it is diagonal. Hence we need only compute the multiplicity of zero in $T$.  The relevant calculation is
\[\scalebox{.85}{$\displaystyle{
	p(\lambda) = \det\begin{pmatrix} \lambda+\alpha_1 + \mu_1 & 0 & 0 & \ldots & 0 & 0 & -fc_v\\
	-\alpha_1 & \lambda+\alpha_2+\mu_2 & 0 & \ldots & 0 & 0 & 0 \\
	0 & -\alpha_2 & \lambda+\alpha_3+\mu_3 & \ldots  & 0 & 0 & 0\\
	\vdots &\vdots &\vdots & \ddots & \vdots &\vdots &\vdots \\
	0 & 0 & 0& \ldots & -\alpha_{j-1} & \lambda + \gamma + \mu_j& 0 \\
	-\frac{fc\bar{S}_v}{\bar{S}} &	-\frac{fc\bar{S}_v}{\bar{S}} &	-\frac{fc\bar{S}_v}{\bar{S}} & \ldots &-\frac{fc\bar{S}_v}{\bar{S}}&-\frac{fc\bar{S}_v}{\bar{S}}
	&\lambda + \tilde{\mu}
\end{pmatrix}}$}
\]
First we expand along the top row
\[\scalebox{.9}{$\displaystyle{
	p(\lambda) = (\lambda + \alpha_1 + \mu_1) \det \begin{pmatrix}
	 \lambda+\alpha_2+\mu_2 & 0 & \ldots & 0 & 0 & 0 \\
	 -\alpha_2 & \lambda+\alpha_3+\mu_3 & \ldots  & 0 & 0 & 0\\
	\vdots &\vdots & \ddots & \vdots &\vdots &\vdots \\
	 0 & 0& \ldots & -\alpha_{j-1} & \lambda + \gamma + \mu_j& 0 \\
	-\frac{fc\bar{S}_v}{\bar{S}} &	-\frac{fc\bar{S}_v}{\bar{S}} & \ldots &-\frac{fc\bar{S}_v}{\bar{S}}&-\frac{fc\bar{S}_v}{\bar{S}}
	&\lambda + \tilde{\mu}
\end{pmatrix}}$}
\]
\[\scalebox{.9}{$\displaystyle{
	 + (-1)^{j} (-fc_v) \det \begin{pmatrix}
	-\alpha_1 & \lambda+\alpha_2+\mu_2 & 0 & \ldots & 0 & 0  \\
	0 & -\alpha_2 & \lambda+\alpha_3+\mu_3 & \ldots  & 0 & 0 \\
	\vdots &\vdots &\vdots & \ddots & \vdots &\vdots  \\
	0 & 0 & 0& \ldots & -\alpha_{j-1} & \lambda + \gamma + \mu_j \\
	-\frac{fc\bar{S}_v}{\bar{S}} &	-\frac{fc\bar{S}_v}{\bar{S}} &	-\frac{fc\bar{S}_v}{\bar{S}} & \ldots &-\frac{fc\bar{S}_v}{\bar{S}}&-\frac{fc\bar{S}_v}{\bar{S}}
 \end{pmatrix}}$}
\]
The determinant of the first matrix is the product of the diagonals, being that it is lower triangular.
Computing the second determinant requires some courage, and a little trickiness.  First we divide the bottom row by $-\frac{fc\bar{S}_v}{\bar{S}}$.
For the value of the determinant to stay the same, we also multiply it by the same amount.  Hence we compute
\[
	(-1)^{j} \frac{f^2 c c_v\bar{S}_v}{\bar{S}} \det \begin{pmatrix}
	-\alpha_1 & \lambda+\alpha_2+\mu_2 & 0 & \ldots & 0 & 0  \\
	0 & -\alpha_2 & \lambda+\alpha_3+\mu_3 & \ldots  & 0 & 0 \\
	\vdots &\vdots &\vdots & \ddots & \vdots &\vdots  \\
	0 & 0 & 0& \ldots & -\alpha_{j-1} & \lambda + \gamma + \mu_j \\
			1 &1 &1& \ldots  &1&1
	\end{pmatrix}
\]
by expanding along the bottom row.  Then at each step of the expansion we will take the determinant of a block diagonal
matrix, and each matrix along the diagonal will be a triangular matrix.  In particular, once the $k$th column and the bottom row are removed, the diagonal to the left of the column has
 the elements $-\alpha_1, \ldots , -\alpha_{k-1}$ with the only other nonzero elements above this diagonal.
To the right of the column the diagonal elements are $\lambda + \alpha_{k+1} + \mu_{k+1}, \ldots ,
\lambda + \gamma + \mu_j$, and on this side the only other nonzero elements are above the diagonal.  Thus the minors can be written in this form
\[
\begin{pmatrix} A & 0\\0 &B \end{pmatrix}
\]
According to \cite{matmat} the determinant of this minor will be the product of the determinants of the diagonal matrices, and as we
have already mentioned $A$ is upper triangular and $B$ is lower triangular.  Thus, the determinant of the minor is the product of the diagonal
elements.  The signs for the minors along the bottom are  $(-1)^{j+k}$.  Furthermore we will multiply the $k-1$ negative elements in the upper matrix.
Hence the sign of each term is $(-1)^{j+k+k-1} = (-1)^{j+2k-1} = (-1)^{j-1}$.  None of these depend on $k$, we can factor it out and combine it with
the $(-1)^j$ on the outside, and have $(-1)^{j+j-1} = -1$.  Thus, the characteristic polynomial takes the form
\[
	\begin{split}
		p(\lambda) = & (\lambda + \alpha_1 + \mu_1) \ldots (\lambda + \gamma + \mu_j)(\lambda + \tilde{\mu})\\
			     & - \frac{f^2 c c_v \bar{S}_v}{\bar{S}} \left[ (\lambda + \alpha_2 + \mu_2) \ldots (\lambda + \gamma + \mu_j) \right.\\
		      &+ \alpha_1(\lambda + \alpha_3 + \mu_3)
	\ldots (\lambda + \gamma + \mu_j) \\
& \quad \left. + \ldots \alpha_1 \ldots \alpha_k( \lambda + \alpha_{k+2} + \mu_{k+2}) \ldots(\lambda + \gamma + \mu_j) + \ldots + \alpha_1 \ldots \alpha_{j-1} \right]
 \end{split}
\]
To show that $0$ is a simple eigenvalue, we must show that when $R_0 = 1$, the constant term of this polynomial is 0 and the linear term is nonzero.
 In order to ease some of the calculation we next
build up some notation.  For indexing purposes define $\alpha_0=1$.  Now let $\xi_i = \alpha_i + \mu_i$ for $1 \leq i \leq j-1$ and $\xi_j = \gamma + \mu_j$ and $\xi_{j+1}
= \tilde{\mu}$.  This sets up a consistent notation for the parameters.  Also, we need to take products of all but one of these parameters, so
we define the following
\[
	\xi_1 \ldots \hat{\xi}_i \ldots \xi_{j} = \xi_1 \ldots \xi_{i-1} \xi_{i+1} \ldots \xi_j
\]
The hat tells us which of the parameters is deleted from the product.  We can now rewrite $R_0$ using this notation
\[
	R_0 = f\sqrt{ \frac{c c_v \bar{S}_v}{\bar{S}\xi_{j+1}}\sum_{k=1}^j \frac{\alpha_0 \ldots \alpha_{k-1}}{\xi_1 \ldots \xi_k}}
\]
and $p(\lambda)$ becomes
\begin{equation} \label{charpoly}
	p(\lambda) = \prod_{i=1}^j (\lambda + \xi_i) - \frac{f^2 c c_v \bar{S}_v}{\bar{S}} \sum_{i=0}^{j-1} \alpha_0 \ldots \alpha_{i}(\lambda + \xi_{i+2}) \ldots (\lambda + \xi_{j} )	
\end{equation}
The constant term is found by evaluating $p(0)$:
\[
	p(0) = \prod_{i=1}^j \xi_i  -  \frac{f^2 c c_v \bar{S}_v}{\bar{S}} \sum_{i=0}^{j-1} \alpha_0 \ldots \alpha_{i}\xi_{i+2} \ldots \xi_{j}
\]
Next, we solve $p(0) = 0$.
\[
	1 = \frac{f^2 c c_v \bar{S}_v}{\bar{S}} \frac{1}{\xi_1 \ldots \xi_j\xi_{j+1}}\sum_{i=0}^{j-1} \alpha_0 \ldots \alpha_{i}\xi_{i+2} \ldots \xi_{j}
\]
\[
	 =  \frac{f^2 c c_v \bar{S}_v}{\bar{S}\xi_{j+1}}\sum_{i=0}^{j-1} \frac{\alpha_0 \ldots \alpha_{i}\xi_{i+2} \ldots \xi_{j}}{\xi_1 \ldots \xi_j}
\]
\[
	=  \frac{f^2 c c_v \bar{S}_v}{\bar{S}\xi_{j+1}}\sum_{i=0}^{j-1}\frac{\alpha_0 \ldots \alpha_{i}}{\xi_1 \ldots \xi_{i+1}}
\]
Letting $k=i+1$ this becomes
\[
	 =   \frac{f^2 c c_v \bar{S}_v}{\bar{S}\xi_{j+1}}\sum_{k=1}^{j}\frac{\alpha_0 \ldots \alpha_{k-1}}{\xi_1 \ldots \xi_{k}} = R_0^2
\]
This implies that the linear term is 0 if and only if $R_0=1$.  This only establishes that the 0 has nontrivial algebraic multiplicity,
to establish simplicity we check that the coefficient of the linear term is nonzero when $R_0=1$.   Given a product of factors
\[
	(\lambda+a_1) \ldots (\lambda+a_n)
\]
the coefficient of the linear term, in the ``hat" notation, is
\[
	\sum_{i=1}^n a_1\ldots \hat{a}_i \ldots a_n
\]
Applying this to \eqref{charpoly} we find the coefficient of the linear term of the characteristic polynomial to be
\[
	\sum_{i=1}^{j+1} \xi_1 \ldots \hat{\xi}_i \ldots \xi_{j+1} - \frac{f^2 c c_v \bar{S}_v}{\bar{S}}\sum_{k=0}^{j-2} \alpha_0 \ldots \alpha_k \sum_{i=k+2}^ j
	\xi_{k+2} \ldots \hat{\xi}_i \ldots \xi_j
\]
Assume that this is equal to zero when $R_0=1$ and we will arrive at a contradiction.  Multiplying the second term by $\frac{\xi_{j+1}}{\xi_{j+1}}$
\[
	0= \sum_{i=1}^{j+1} \xi_1 \ldots \hat{\xi}_i \ldots \xi_{j+1} - \frac{f^2 c c_v \bar{S}_v}{\bar{S}\xi_{j+1}}\sum_{k=0}^{j-2} \sum_{i=k+2}^ j \alpha_0 \ldots \alpha_k
	\xi_{k+2} \ldots \hat{\xi}_i \ldots \xi_j \xi_{j+1}
\]
where upon some manipulation the expression yields
\[
	1 = \frac{\displaystyle{ \frac{f^2 c c_v \bar{S}_v}{\bar{S}\xi_{j+1}}\sum_{k=0}^{j-2} \sum_{i=k+2}^ j \alpha_0 \ldots \alpha_k
	\xi_{k+2} \ldots \hat{\xi}_i \ldots \xi_j \xi_{j+1}}}{\displaystyle{\sum_{i=1}^{j+1} \xi_1 \ldots \hat{\xi}_i \ldots \xi_{j+1}}}
\]
But $R_0^2=1$ so
\[
	\frac{f^2 c c_v \bar{S}_v}{\bar{S}\xi_{j+1}}\sum_{k=1}^{j}\frac{\alpha_0 \ldots \alpha_{k-1}}{\xi_1 \ldots \xi_{k}} =
	\frac{\displaystyle{ \frac{f^2 c c_v \bar{S}_v}{\bar{S}\xi_{j+1}}\sum_{k=0}^{j-2} \sum_{i=k+2}^ j \alpha_0 \ldots \alpha_k
	\xi_{k+2} \ldots \hat{\xi}_i \ldots \xi_j \xi_{j+1}}}{\displaystyle{\sum_{i=1}^{j+1} \xi_1 \ldots \hat{\xi}_i \ldots \xi_{j+1}}}
\]
Canceling the constant in front, and multiplying by the denominator we get
\[
	\sum_{i=1}^j \sum_{k=1}^{j+1} \frac{\alpha_0 \ldots \alpha_{i-1} \xi_{1} \ldots \hat{\xi}_k \ldots \xi_{j+1}}{\xi_1 \ldots \xi_i}
	 = \sum_{k=0}^{j-2} \sum_{i=k+2}^ j \alpha_0 \ldots \alpha_k
	\xi_{k+2} \ldots \hat{\xi}_i \ldots \xi_j \xi_{j+1}
\]
We can split the first sum into two parts depending on the largest value of $k$.  In particular, when $k \leq i$  $\xi_k$ will not cancel out of the denominator,
but when $k \geq i+1$ the whole denominator will cancel, so we write
\[
	\sum_{i=1}^j \sum_{k=1}^i \frac{\alpha_0 \ldots \alpha_{i-1} \xi_{i+1} \ldots \xi_{j+1}}{\xi_k} + \sum_{i=1}^j \sum_{k=i+1}^j \alpha_0 \ldots \alpha_{i-1}
	\xi_{i+1} \ldots \hat{\xi}_k \ldots \xi_{j+1}
\]
\[
	= \sum_{k=0}^{j-2} \sum_{i=k+2}^ j \alpha_0 \ldots \alpha_k
	\xi_{k+2} \ldots \hat{\xi}_i \ldots \xi_j \xi_{j+1}
\]
Exchanging $i$ and $k$ on the right hand side this becomes
\[
	\sum_{i=1}^j \sum_{k=1}^i \frac{\alpha_0 \ldots \alpha_{i-1} \xi_{i+1} \ldots \xi_{j+1}}{\xi_k} + \sum_{i=1}^j \sum_{k=i+1}^j \alpha_0 \ldots \alpha_{i-1}
	\xi_{i+1} \ldots \hat{\xi}_k \ldots \xi_{j+1}
\]
\[
	= \sum_{i=0}^{j-2} \sum_{k=i+2}^ j \alpha_0 \ldots \alpha_i
	\xi_{i+2} \ldots \hat{\xi}_k \ldots \xi_j \xi_{j+1}
\]
Shifting the $i$ index up by 1 on the RHS yields
\[
	\sum_{i=1}^j \sum_{k=1}^i \frac{\alpha_0 \ldots \alpha_{i-1} \xi_{i+1} \ldots \xi_{j+1}}{\xi_k} + \sum_{i=1}^j \sum_{k=i+1}^j \alpha_0 \ldots \alpha_{i-1}
	\xi_{i+1} \ldots \hat{\xi}_k \ldots \xi_{j+1}
\]
\[
	= \sum_{i=1}^{j-1} \sum_{k=i+1}^ j \alpha_0 \ldots \alpha_{i-1}
	\xi_{i+1} \ldots \hat{\xi}_k \ldots \xi_j \xi_{j+1}
\]
Thus we have that
\[
	\sum_{i=1}^j \sum_{k=1}^i \frac{\alpha_0 \ldots \alpha_{i-1} \xi_{i+1} \ldots \xi_{j+1}}{\xi_k} + \alpha_0 \ldots \alpha_{j-1} = 0
\]
However, this cannot be so because all of the rates are positive.  Hence a contradiction and the conclusion that the linear term cannot be 0 when $R_0=1$, and  $0$ is a simple eigenvalue of $Df(x_0)$.
\subsection{Proof of Lemma 1}
\label{lem1proof}
First note that the last three components of $v$ are 0.  This follows from the fact that
\[
J_4 = \begin{pmatrix} \mu_s &0&0\\0& \widetilde{\mu}_s&0\\ 0&0&\mu_r \end{pmatrix}
\]
is invertible.
Because
\[
	f = \frac{\mu + 1}{\zeta},
\]
the only nonzero derivatives in the expression for $b$ are in the $I$ and $\widetilde{I}$ compartments, since these are the only ones that involve $\mu$.
As the last three components of $v$ are 0, we have
\[
	b = v_1\left(\sum_{k=1}^jw_k\frac{\partial f_1}{\partial I_k \partial \mu} + w_{j+1} \frac{\partial f_1}{\partial \widetilde{I}\partial \mu}
		+ w_{j+1}\frac{\partial f_1}{\partial S \partial \mu}  + w_{j+3}\frac{\partial f_1}{\partial \widetilde{S}\partial \mu}  +
	w_{j+4}\frac{\partial f_1}{\partial R \partial \mu} \right)
\]
\[
	+ v_{j+1} \left(\sum_{k=1}^jw_k\frac{\partial f_{j+1}}{\partial I_k \partial \mu} + w_{j+1} \frac{\partial f_{j+1}}{\partial \widetilde{I}\partial \mu}
		+ w_{j+2}\frac{\partial f_{j+1}}{\partial S \partial \mu}  + w_{j+3}\frac{\partial f_{j+1}}{\partial \widetilde{S}\partial \mu}  +
	w_{j+4}\frac{\partial f_{j+1}}{\partial R \partial \mu} \right)
\]
Taking derivatives and evaluating at the DFE gives that
\[
	b = v_1(0 + w_{j+1}\frac{c_v}{\zeta} + 0 + 0 + 0) + v_{j+1}\left(\sum_{k=1}^j w_k \frac{c\overline{S}_v}{\overline{S}\zeta} + 0 + 0 + 0 + 0 \right)
\]
\[
	= v_1w_{j+1} \frac{c_v}{\zeta} + \frac{v_{j+1}c\overline{S}_v}{\overline{S}\zeta} \sum_{k=1}^j w_k
\]
We need to show that this is nonzero. First we claim that at least one of the $w_{i}$ are nonzero.  Suppose not, and that we have $w_1 = \ldots = w_{j+1} = 0$.
Then since $v_{j+2} = v_{j+3} = v_{j+4} = 0$,
\[
	vw = \sum_{i=1}^{j+4} v_iw_i = 0
\]
which is a contradiction.  Next we claim that $v_1 \neq 0$ or $v_{j+1} \neq 0$.  Suppose to the contrary that $v_1 = v_{j+1} = 0$. Then because
\[
	(v_1, \ldots, v_{j+1}, 0,0,0) Df(x_0) = 0
\]
it follows that
\[
	(0, v_2, \ldots, v_j, 0,0,0,0) Df(x_0) = 0.
\]
\[
 	\Rightarrow (v_2, \ldots, v_j) \begin{pmatrix} \alpha_1 & -\alpha_2-\mu_2 & \ldots & 0 & 0\\
	\vdots & \ddots & \ddots & \vdots & \vdots\\
	0 & 0 & \vdots & -\alpha_{j-1} - \mu_{j-1} & 0\\
0 &0& \vdots & \alpha_{j-1} & -\gamma-\mu_j \end{pmatrix} = 0
\]
The last column shows that $v_j = 0$, so the expression further reduces to
\[
	(v_2, \ldots, v_{j-1}) \begin{pmatrix} \alpha_1 & -\alpha_2-\mu_2 & \ldots & 0 \\
	\vdots & \ddots & \ddots & \vdots\\
	0 & 0 & \vdots & -\alpha_{j-1} - \mu_{j-1}  \end{pmatrix} = 0
\]
Again, the last column implies that $v_{j-1} = 0$, and so on.  This means  $v=0$, but $vw=1$, which is a contradiction.  Then because all the terms
in $b$ are nonnegative, with some of them being nonzero, we arrive at $b \neq 0$.

To show that $a>0$ note that
\[
	J_3 = \begin{pmatrix} 0&\ldots &0&fc_v\\ \frac{fc\overline{S}_v}{\overline{S}} & \ldots & \frac{fc\overline{S}_v}{\overline{S}} &0\\
0 &\ldots& \gamma & 0 \end{pmatrix}
\]
and
\[
J_4^{-1} = \begin{pmatrix} \frac{1}{\mu_s} & 0 & 0\\ 0 & \frac{1}{\widetilde{\mu}_s} & 0\\ 0 & 0 & \frac{1}{\mu_r} \end{pmatrix}
\]
Thus
\[
	-J_4^{-1} J_3 = \begin{pmatrix} 0&\ldots &0&\frac{fc_v}{\mu_s}\\ \frac{fc\overline{S}_v}{\widetilde{\mu}_s\overline{S}}
					 & \ldots & \frac{fc\overline{S}_v}{\widetilde{\mu}_s\overline{S}} &0\\
0 &\ldots& \frac{\gamma}{\mu_{r}} & 0 \end{pmatrix}
\]
We can list the nonzero elements of this matrix : $\varepsilon_{1(j+1)}, \varepsilon_{21}, \ldots, \varepsilon_{2j}, \varepsilon_{3j}$.
The second derivatives of the infected components with respect to an infected variable are all zero, since one differentiation removes all of that
infected variable.  When differentiating with respect to an infected variable and an uninfected variable, the derivatives of the $I_2$ through $I_j$
components will be zero as they contain no uninfected variables.  The nonzero derivatives are
\[
	\frac{\partial^2 f_1}{\partial \widetilde{I}\partial S} = \frac{fc_v}{\overline{S}}
\]
\[
	\frac{\partial^2 f_{j+1}}{\partial I_k \partial \widetilde{S}} = \frac{fc}{\overline{S}}
\]
So then
\[
	a= \sum_{i,j,k = 1}^{j+1} v_iw_jw_k \left( 0 - \frac{f^2c_v^2}{\overline{S}\mu_s} - \sum_{l=1}^j \frac{f^2c^2\overline{S}_v}{\overline{S}^2 \widetilde{\mu}_s}
	\right)
\]
\[
= - \sum_{i,j,k = 1}^{j+1} v_iw_jw_k \left( \frac{f^2c_v^2}{\overline{S}\mu_s}  + \frac{f^2c^2\overline{S}_v}{\overline{S}^2 \widetilde{\mu}_s} ]\right)
\]
As we have already shown, the $v_i, w_i \geq 0$, and since the parameters are positive $a<0$.
\end{document}